\documentclass{daj}

\usepackage{amsmath}
\usepackage{amssymb}
\usepackage{amsthm}

\usepackage{pdfsync}

\numberwithin{equation}{section}

\theoremstyle{plain}
\newtheorem{theorem}{Theorem}[section]
\newtheorem{corollary}[theorem]{Corollary}
\newtheorem{prop}[theorem]{Proposition}
\newtheorem{lemma}[theorem]{Lemma}

\theoremstyle{remark}
\newtheorem{remark}[theorem]{Remark}

\theoremstyle{definition}
\newtheorem{definition}[theorem]{Definition}

\newcommand{\supp}{\textrm{Supp}}

\newcommand{\la}{\langle}
\newcommand{\ra}{\rangle}

\newcommand{\e}{\varepsilon}
\newcommand{\N}{\mathbb{N}}
\newcommand{\R}{\mathbb{R}}
\newcommand{\Z}{\mathbb{Z}}
\newcommand{\dist}{\mathrm{dist}}

\newcommand{\cP}{\mathcal{P}}
\newcommand{\cM}{\mathcal{M}}

\newcommand{\cN}{\mathcal{N}}
\newcommand{\cD}{\mathcal{D}}
\newcommand{\cG}{\mathcal{G}}

\newcommand{\cE}{\mathcal{E}}
\newcommand{\cF}{\mathcal{F}}

\DeclareMathOperator{\hdim}{dim_H}
\DeclareMathOperator{\pdim}{dim_P}

\DeclareMathOperator{\ledim}{\underline{\dim}_e}
\DeclareMathOperator{\uedim}{\overline{\dim}_e}
\DeclareMathOperator{\lbdim}{\underline{\dim}_B}
\DeclareMathOperator{\ubdim}{\overline{\dim}_B}
\DeclareMathOperator{\mlbdim}{\underline{\dim}_{MB}}

\newcommand{\wh}{\widehat}
\newcommand{\wt}{\widetilde}

\dajAUTHORdetails{%
  title = {On Distance Sets, Box-Counting and Ahlfors Regular Sets}, 
  author = {Pablo Shmerkin},
  plaintextauthor = {Pablo Shmerkin},
  keywords = {distance sets, box dimension, Ahlfors regular sets, CP-processes},
}   

\dajEDITORdetails{%
   year={2017},
   number={9},
   received={16 January 2017},   
   published={23 May 2017},  
   doi={10.19086/da.1643},       
}   

\begin{document}

\begin{frontmatter}[classification=text]

\author[ps]{Pablo Shmerkin \thanks{P.S. was partially supported by Projects PICT 2013-1393 and PICT 2014-1480 (ANPCyT)}}

\begin{abstract}
We obtain box-counting estimates for the pinned distance sets of (dense subsets of) planar discrete Ahlfors-regular sets of exponent $s>1$. As a corollary, we improve upon a recent result of Orponen, by showing that if $A$ is Ahlfors-regular of dimension $s>1$, then almost all pinned distance sets of $A$ have lower box-counting dimension $1$. We also show that if $A,B\subset\R^2$ have Hausdorff dimension greater than $1$ and $A$ is Ahlfors-regular, then the set of distances between $A$ and $B$ has modified lower box-counting dimension $1$, which taking $B=A$ improves Orponen's result in a different direction, by lowering packing dimension to modified lower box-counting dimension. The proofs involve ergodic-theoretic ideas, relying on the theory of CP-processes and projections.
\end{abstract}

\end{frontmatter}


\section{Introduction and main results}

In 1985, Falconer \cite{Falconer85} (implicitly) conjectured that if $A\subset\R^d$, with $d\ge 2$, is a Borel set of Hausdorff dimension at least $d/2$, then the set of distances
\[
\dist(A,A) = \{ |x-y|:x,y\in A\}
\]
has Hausdorff dimension $1$. He also showed that the value $d/2$ would be sharp. The conjecture remains wide open in every dimension, but several deep advances have been obtained; we discuss in some detail what is known in the plane, and refer to \cite{Erdogan05} for some of the known results in higher dimensions.  Throughout the paper, $\hdim$, $\pdim$, $\lbdim$, $\ubdim$, and $\mlbdim$ denote, respectively Hausdorff, packing, lower box-counting, upper box-counting, and modified lower box-counting dimensions. See \S\ref{subsec:dimension} for the definitions, and \cite{Falconer14} for further background on fractal dimensions.

Relying on earlier work of Mattila \cite{Mattila87}, and using deep harmonic-analytic techniques, Wolff \cite{Wolff99} showed that if $\hdim(A)>4/3$, then $\dist(A,A)$ has positive Lebesgue measure. As remarked in \cite{Wolff99}, $4/3$ appears to be the limit of these methods. Nevertheless, Iosevich and Liu \cite{IosevichLiu16} recently obtained an improvement for a large class of cartesian products.

Assuming only that $\hdim(A)\ge 1$, Bourgain \cite{Bourgain03} (relying on earlier work of Katz and Tao \cite{KatzTao01}) used sophisticated  additive-combinatorial arguments to prove that
\[
\hdim(\dist(A,A)) > 1/2+\varepsilon,
\]
for some small absolute constant $\varepsilon>0$.

To the best of our knowledge, the following stronger version of Falconer's conjecture might hold: if $\hdim(A)\ge 1$, then there exists $x\in A$ such that the \textbf{pinned distance set}
\[
\dist(x,A) = \{ |x-y|:y\in A\}
\]
has Hausdorff dimension $1$.  Although the method of Wolff does not appear to say anything about pinned distance sets, Peres and Schlag \cite{PeresSchlag00} employed the transversality method to prove that, under the stronger assumption $\hdim(A)>3/2$, for all $x$ outside of a set of dimension at most $3-\hdim(A)$, the pinned distance set $\dist(x,A)$ has positive Lebesgue measure.

Very recently, Orponen \cite{Orponen17} approached the problem from a different angle. Recall that a set $A\subset\R^d$ is called \textbf{$(s,C)$-Ahlfors regular}, or $s$-Ahlfors regular with constant $C$, if there exists a measure $\mu$ supported on $A$, such that $C^{-1} r^s \le \mu(B(x,r)) \le C r^s$ for all $x\in \supp(\mu)$ and all $r\in (0,1]$. Orponen showed that if $A\subset\R^2$ is $(s,C)$-Alhfors regular for some $s\ge 1$ and any $C>1$, then the \emph{packing} dimension of $\dist(A,A)$ is $1$. In fact, a small modification of his method shows that also the lower box-counting of $\dist(A,A)$ equals $1$.

In this article, we improve upon Orponen's result in several directions: we obtain results on the existence of many large \emph{pinned} distance sets, we weaken slightly the hypothesis of Ahlfors-regularity, we show that the \emph{modified} lower box-counting dimension of the distance set is 1, and we are able to consider the set of distances between two different sets.

Our first main result is a discretized version for large subsets of (discrete) Ahlfors-regular sets: we say that a set $A\subset \R^d$ is \textbf{discrete $(s,C)$-Ahlfors regular at scale $2^{-N}$} if
\[
C^{-1} 2^{(N-k)s} \le |B(x,2^{-k})\cap A| \le C 2^{(N-k)s}\quad\text{for all }x\in A, k\in [N],
\]
where $[N]=\{0,1,\ldots,N-1\}$. For a bounded set $F\subset \R^d$, we denote by $\cN(F,\e)$ the number of $\e$-grid cubes hit by $F$.

\begin{theorem} \label{thm:many-large-pinned-dist-sets}
Given $s>1, C>1, t\in (0,1)$, there exist $\e=\e(s,C,t)>0$ and $N_0=N_0(s,C,t)\in\N$ such that the following holds:

If $N\ge N_0$, and $A\subset [0,1]^2$ is a subset of a discrete $(s,C)$-Ahlfors regular set at scale $2^{-N}$, then
\[
| x\in A:\cN(\dist(x,A),2^{-N}) < 2^{tN}| \le 2^{(s-\e)N}.
\]
\end{theorem}
An inspection of the proof shows that we can take $\e=(1-t)/C'$ for some effective $C'=C'(s,C)>0$. The value of $N_0$ does not appear to be effective from the current proof.

Theorem \ref{thm:many-large-pinned-dist-sets} fails rather dramatically for $s=1$, as witnessed by the example described in \cite[Eq. (2) and Figure 1]{KatzTao01}. Namely, given $N\gg 1$, let
\[
A_N = \left\{ (x,y): x\in \{ i 2^{-N/2}:0\le i < 2^{N/2}-1 \}, y\in  \{ j 2^{-N}: 0\le j < 2^{N/2} \}\right\}.
\]
This is a discrete $1$-Ahlfors regular set at scale $2^{-N}$, yet one can check that $\cN(\dist(x,A_N),2^{-N}) = O(2^{N/2})$ for all $x\in A_N$. In the proof of Theorem \ref{thm:many-large-pinned-dist-sets}, the role of the assumption $s>1$ is to ensure that the set of directions determined by pairs of points in $A$ is dense ``with high multiplicity'', see \S\ref{subsec:conical-density} below. This obviously fails for each of the sets $A_N$ and, more generally, for many discrete $1$-Ahlfors regular sets. We thank an anonymous referee for pointing out this ``almost counter-example'' to Theorem \ref{thm:many-large-pinned-dist-sets}.

We obtain several corollaries from Theorem \ref{thm:many-large-pinned-dist-sets} . Firstly, for sets of full Hausdorff dimension inside an Ahlfors-regular set, nearly all pinned distance sets have full lower box-counting dimension:
\begin{corollary} \label{cor:pinned-dist-set-large-dev}
For every $t\in (0,1)$, $s>1$, $C>0$ there is $\e=\e(s,C,t)>0$ such that the following holds. Let $A$ be a bounded subset of a $(s,C)$-Ahlfors regular set in $\R^2$. Then
\[
\hdim\{x\in A: \lbdim(\dist(x,A))<t\} < s-\e.
\]
Moreover, if $\mathcal{H}^s(A)>0$, then
\[
\lbdim(\dist(x,A))=1 \quad\text{for $\mathcal{H}^s$-almost all $x\in A$}.
\]
In particular, this holds if $A$ is itself $(s,C)$-Ahlfors regular.
\end{corollary}
In the above corollary, $\mathcal{H}^s$ denotes $s$-dimensional Hausdorff measure. It is also easy to deduce a statement purely about box-counting dimensions:
\begin{corollary} \label{cor:dist-set-box-dim}
Let $A$ be a bounded subset of a $(s,C)$-Ahlfors regular set in $\R^2$, with $\lbdim(A)=s>1$ (resp. $\ubdim(A)=s>1$). Then
\[
\lbdim(\dist(A,A)) = 1 \quad(\text{resp. } \ubdim(\dist(A,A)=1)).
\]
\end{corollary}
We underline that the Hausdorff dimension of sets satisfying the hypothesis of the above corollary may be arbitrarily small, or even zero.

Our second main result concerns the set of distances between two, possibly disjoint, sets $A, B\subset\R^2$. Although here we do not get a discretized result, we do get large \emph{modified} lower box-counting dimension of the distance set (which we recall is smaller than both lower box dimension and packing dimension, and unlike the former is countably stable).   Moreover, while for one of the sets we still need to assume Ahlfors-regularity, for the other we only require that the Hausdorff dimension strictly exceeds $1$.
\begin{theorem} \label{thm:mlbdim-distance-sets-AR}
Let $A,B\subset \R^2$ be two Borel sets such that $\hdim(A)>1$ and $B$ is $(s,C)$-Ahlfors regular for some $s>1$. Then
\[
\mlbdim(\dist(A,B)) = 1.
\]
In particular, if $A$ is $s$-Ahlfors regular with $s>1$, then its distance set has full modified lower box-counting dimension.
\end{theorem}
In fact, we are able to somewhat weaken the assumptions on $A$ and $B$, see Theorem \ref{thm:mlbdim-distance-sets} below and the remark after the proof.

The proof of Theorem \ref{thm:mlbdim-distance-sets-AR} also yields the following:
\begin{corollary} \label{cor:pinned-dist-set-upper-box-dim}
Let $A,B\subset \R^2$ be two Borel sets such that $\hdim(A)>1$ and $B$ is $(s,C)$-Ahlfors regular for some $s>1$.  Then
\[
\hdim(\{ x\in A: \ubdim(\dist(x,B))<1 \})\le 1.
\]
In particular, this applies to $A=B$.
\end{corollary}
Compared with Corollary \ref{cor:pinned-dist-set-large-dev}, we lower the size of the exceptional set (from zero measure to Hausdorff dimension $1$), at the price of dealing with upper box-counting dimension instead of lower box-counting dimension.

For the proofs, we follow some of the ideas of Orponen \cite{Orponen17}, but there are substantial differences. A key step in his approach is a projection theorem for entropy in the Ahlfors regular case, see \cite[Proposition 3.8]{Orponen17}, which is applied at all scales. It is unclear whether such a result continues to hold after removing even very small pieces of the initial regular set. Hence, in order to make the method robust under passing to large subsets (which is essential to the proof of Theorem \ref{thm:many-large-pinned-dist-sets}), we needed a different device to handle the entropy of projections. This more flexible device is the theory of CP-processes and projections developed in \cite{HochmanShmerkin12}, which we review in Section \ref{sec:preliminaries}. Very roughly speaking, a CP-process is a measure-valued dynamical system which consists in zooming in dyadically towards a typical point of the measure. Thus, this paper is another example of an application of ergodic-theoretic ideas to problems that, a priori, have nothing to do with dynamics or ergodic theory.

As noted by Orponen already in \cite{Orponen12}, in the study of distance sets the spherical projections $\sigma_x(y)=(x-y)/|x-y|$ play a key role (the reason is that they arise when linearizing the distance function). An important fact in Orponen's approach is that spherical projections of sets of dimension at least $1$ are dense. For the proof of Theorem \ref{thm:many-large-pinned-dist-sets} we require a discrete quantitative version of this (established in \S\ref{subsec:conical-density}), while for Theorem \ref{thm:mlbdim-distance-sets-AR} we rely instead on a recent result of Mattila and Orponen \cite{MattilaOrponen15}, see also \cite{Orponen16}.

The paper is organized as follows. In Section \ref{sec:preliminaries} we set up notation, recall different notions of dimensions, and review the parts of the theory of CP-processes that we will require. In Section \ref{sec:Ahlfors-regularity} we discuss a notion of regularity weaker than Ahlfors-regularity. Theorem \ref{thm:many-large-pinned-dist-sets} and its corollaries are proved in Section \ref{sec:pinned-dist-sets}, while Theorem \ref{thm:mlbdim-distance-sets-AR} is proved in Section \ref{sec:distances-between-sets}.

\section{Preliminaries}
\label{sec:preliminaries}

\subsection{Notation}

We use $O(\cdot)$ notation: $A=O(B)$ means $0\le A\le C B$ for some constant $B$; if $C$ is allowed to depend on any parameters, this are denoted as subscripts; e.g. $A = O_d(B)$ means $0 \le A \le C(d) B$. Finally, $A=\Omega(B)$ means $B=O(A)$, and likewise with subscripts.

Given a metric space $X$, we denote the family of all Borel probability measures on $X$ by $\cP(X)$, and the family of all Radon measures on $X$ by $\cM(X)$. When $X$ is compact, $\cP(X)$ is endowed with the weak topology, which is metrizable. If $f:X\to Y$ and $\mu\in\cM(X)$, the \emph{push-down measure} $f\mu$ is defined as $f\mu(A)=\mu(f^{-1}A)$. We note this is sometimes denoted $f_\#\mu$.

If $\mu\in\cM(X)$ and $\mu(A)>0$, then $\mu|_A$ is the restriction of $\mu$ to $A$ and, provided also $\mu(A)<\infty$, we denote by $\mu_A$ the restriction normalized to be a probability measure, that is
\[
\mu_A(B) =\frac{1}{\mu(A)}\mu(A\cap B).
\]

We work in an ambient dimension $d$; this will always be $1$ or $2$ in this paper.  We denote by $\cD_k^{(d)}$ the partition of $\R^d$ into half-open dyadic cubes
\[
\left\{  [j_1 2^{-k}, (j_1+1)2^{-k})\times \cdots\times [j_d 2^{-k}, (j_d+1)2^{-k}): j_1,\ldots,j_d\in\Z \right\}.
\]
When $d$ is clear from context, we simply write $\cD_k$. If $x\in \R^d$, we denote the unique element of $\cD_k^{(d)}$ containing $x$ by $D_k(x)$. In addition to the Euclidean metric, on $\R^d$ we consider the dyadic metric $\rho$ defined as follows: $\rho(x,y)=2^{-\ell}$, where $\ell=\max\{ k: D_k(x)=D_k(y)\}$.

Logarithms are always to base $2$. We denote Shannon entropy of the probability measure $\mu$ with respect to the finite measurable partition $\cF$ by $H(\mu,\cF)$, and the conditional entropy with respect to the finite measurable partition $\mathcal{G}$ by $H(\mu,\cF|\mathcal{G})$. That is,
\begin{align*}
H(\mu,\cF) &= \sum_{F\in\cF} -\mu(F)\log\mu(F),\\
H(\mu,\cF|\mathcal{G}) &= \sum_{G\in\mathcal{G}:\mu(G)>0} \mu(G) H(\mu_G,\cF).
\end{align*}
Here and below we follow the usual convention $0\cdot \log(0)=0$. We denote by $H_k(\mu)$ the normalized entropy $H(\mu,\cD_k)/k$, and note that if $\mu\in\mathcal{P}([0,1)^d)$, then $0\le H_k(\mu)\le 1$. The following are some standard properties of entropy that will get used in the sequel:
\begin{enumerate}
 \item If $|\cF|\le N$, then $H(\mu,\cF)\le \log N$.
 \item If $\cF,\cG$ are finite partitions such that each element of $\cF$ intersects at most $N$ elements of $\cG$ and vice versa, then
\[
|H(\mu,\cF)-H(\mu,\cG)| \le\log N.
\]
  \item (Concavity of entropy). If $\mu,\nu$ are probability measures, $t\in [0,1]$ and $\cF$ is a finite measurable partition, then
  \[
   H(t\mu+(1-t)\nu,\cF) \ge t H(\mu,\cF)+(1-t) H(\nu,\cF).
  \]

\end{enumerate}

Given two integers $A< B$ we denote $[A,B]=\{A,A+1,\ldots, B-1\}$. When $A=0$, we simply write  $[B]=\{ 0,1,\ldots, B-1\}$.

\subsection{Notions of dimension} \label{subsec:dimension}
In this section we quickly review the notions of dimension of sets and measures we will require. For further background on dimensions of sets, see e.g. Falconer's textbook \cite{Falconer14}, while for dimensions of measures and their relationships, we refer to \cite{FLR02}.

Recall that $\cN(F,\e)$ is the number of $\e$-grid cubes that intersect a bounded set $F\subset\mathbb{R}^d$. The \textbf{upper and lower box-counting dimensions} of $F$ are defined as
\begin{align*}
\lbdim(F) &= \liminf_{\e\downarrow 0} \frac{\log \cN(F,\e)}{-\log\e},\\
\ubdim(F) &= \limsup_{\e\downarrow 0} \frac{\log \cN(F,\e)}{-\log\e}.
\end{align*}
These dimensions are not countably stable. After making them countably stable in the natural way, one gets \textbf{modified lower box-counting dimension} $\mlbdim$ and \textbf{packing dimension}:
\begin{align*}
\mlbdim(F) &= \inf\{ \sup_i \lbdim(F_i): F\subset\cup_i F_i \},\\
\pdim(F) &= \inf\{ \sup_i \ubdim(F_i): F\subset\cup_i F_i \}.
\end{align*}
The inequalities $\hdim(F)\le \mlbdim(F)\le \lbdim(F)\le \ubdim(F)$ and $\mlbdim(F)\le \pdim(F)\le \ubdim(F)$ always hold, while $\lbdim$ and $\pdim$ are not comparable in general.

We move on to dimensions of measures. Let $\mu\in\cP(\R^d)$. The \textbf{lower and upper entropy dimensions} are defined as
\begin{align*}
\ledim(\mu) &= \liminf_{k\to\infty} H_k(\mu),\\
\uedim(\mu) &= \limsup_{k\to\infty} H_k(\mu).
\end{align*}

The \textbf{Hausdorff dimension} of $\mu\in\cM(\R^d)$ is
\[
\hdim(\mu) = \inf\{\hdim(A):\mu(A)>0\}.
\]
We note that this is sometimes called the \emph{lower} Hausdorff dimension. Finally, we recall that $\mu\in\cM(\R^d)$ is called \textbf{exact dimensional} if there exists $s\ge 0$ (the \emph{exact dimension} of $\mu$) such that
\[
\lim_{r\downarrow 0} \frac{\log \mu(B(x,r))}{\log r} = s\quad\text{for $\mu$-almost all }x.
\]

For any $\mu\in\cP(\R^d)$ it holds that
\[
\hdim(\mu) \le \ledim(\mu) \le \uedim(\mu),
\]
with strict inequalities possible, see \cite[Theorem 1.3]{FLR02}. However, for measures of exact dimension $s$, there is an equality $\hdim(\mu)=\uedim(\mu)=s$.

\subsection{Global sceneries, entropy and projections}
In this section we recall some results from \cite{Hochman14, Orponen17} (similar ideas go back to \cite{HochmanShmerkin12}). We write $\delta(\omega)$ or $\delta_\omega$ to denote the point mass at $\omega$ (often $\omega$ will be a measure).

We denote the topological support of $\mu\in\cP([0,1)^d)$ in the $\rho$-metric by $\supp_\rho(\mu)$. Note that $x\in\supp_\rho(\mu)$ if and only if $\mu(D_n(x))>0$ for all $n\in\N$. Given $Q\in\cD_n^{(d)}$, let $T_Q$ be the homothety that maps $Q$ onto $[0,1)^d$, and define
\[
\mu^Q = T_Q\left( \frac{\mu|_{Q}}{\mu(Q)}\right).
\]
If $x\in\supp_\rho(\mu)$, we also write $\mu^{x,n}=\mu^{D_n(x)}$ for short. That is, $\mu^{x,n}$ is the normalized restriction of $\mu$ to $D_n(x)$, renormalized back to the unit cube.

Given $\mu\in\cP([0,1)^d)$, $x\in\supp_\rho(\mu)$, and an integer interval $[A,B]$, we write
\begin{align*}
\la \mu,x\ra_{[A,B]} &= \frac{1}{B-A} \sum_{n=A}^{B-1} \delta(\mu^{x,n}),\\
\la \mu \ra_n &= \int \delta(\mu^{x,n})\,d\mu(x) = \sum_{Q\in\cD_n} \mu(Q) \delta(\mu^Q),\\
\la \mu \ra_{[A,B]} &= \int \la \mu,x\ra_{[A,B]} \,d\mu(x) = \frac{1}{B-A} \sum_{n=A}^{B-1} \la\mu\ra_n.
\end{align*}
The second equality in the last line follows from interchanging the order of sum and integration; it will be convenient to alternatively use either definition of $\la \mu \ra_{[A,B]}$.

The following simple but important fact is proved in \cite[Lemma 3.4]{Hochman14}. It allows to recover the global entropy of a measure from local entropies.
\begin{lemma} \label{lem:local-entropy-from-global-entropy}
Let $\mu\in\cP([0,1)^d)$. Then
\[
\left|H_N(\mu) - \int H_q(\eta)\, d\la \mu\ra_{[0,N]}(\eta) \right| = O_d(q/N).
\]
\end{lemma}
In the above lemma, one should think that the value of $q$ is fixed, and $N$ tends to infinity (possibly along a subsequence).

The following is a variant of a result of Orponen \cite{Orponen17}, which in turn adapts ideas of Hochman \cite{Hochman14}.
\begin{prop} \label{prop:projected-entropies-from-local-entropies}
Fix $2< q< N$. Let $\mu\in\cP([0,1)^2)$, and let $U$ be an open set containing $\supp(\mu)$. Suppose that $f: U\to \R$ is a $C^1$ map such that, for some fixed $v\in S^1$,
\[
\|Df(x)-v\| \le 2^{-q} \quad\text{for all } x\in\supp(\mu).
\]
Then, if $\Pi_v(x)= v \cdot x $ denotes the orthogonal projection of $x$ onto a line in direction $v$,
\[
H_N(f\mu) \ge \int  H_q(\Pi_v\eta) \,d\la \mu\ra_{[0,N]}(\eta) - O(q/N) - O(1/q).
\]
The constants in the $O$ notation are absolute.
\end{prop}
\begin{proof}
Orponen \cite[Lemma 3.5]{Orponen17} showed that if $1\le q<N$ and $\nu\in\cP(U)$, then
\begin{equation} \label{eq:projected-entropies-0}
H_N(f\nu) \ge \frac{1}{N} \sum_{\ell=0}^{\lfloor N/q\rfloor -1} \sum_{D\in\cD_{\ell q}^{(2)}} \nu(D) \, H(f\nu_D,\cD_{(\ell+1)q}|\cD_{\ell q}),
\end{equation}
where the sum runs over $D$ with $\nu(D)>0$.

By concavity of entropy, if $\wt{D}\in \cD_j^{(2)}$, then
\[
H_N(f\mu) \ge  \mu(\wt{D}) H_N(f\mu_{\wt{D}}).
\]
Applying \eqref{eq:projected-entropies-0} to each $\nu=\mu_{\wt{D}}$ with $\wt{D}\in\cD_j^{(2)}$, $j\in [q]$, and adding up, and then averaging over $j$, we get
\begin{equation} \label{eq:projected-entropies-1}
H_N(f\mu) \ge \frac{1}{N} \sum_{i=0}^{N-q} \sum_{D\in\cD_i^{(2)}} \mu(D) \,\frac{1}{q} H(f\mu_D,\cD_{i+q}|\cD_i).
\end{equation}
On the other hand, by \cite[Lemma 3.12]{Orponen17}, the almost linearity hypothesis on $f$ ensures that
\begin{equation}  \label{eq:projected-entropies-2}
|H(f\mu_D,\cD_{i+q}|\cD_i) - H(\Pi_v\mu_D,\cD_{i+q}|\cD_i)| = O(1)
\end{equation}
for any $D\in\cD_i^{(2)}$. (This was stated in \cite{Orponen17} for $i$ a multiple of $q$ but the proof in general is identical.)

Finally, as observed in \cite[Remark 3.6]{Orponen17}, the linearity of $\Pi_v$ implies that
\begin{equation} \label{eq:projected-entropies-3}
\frac{1}{q} H(\Pi_v\mu_D,\cD_{i+q}|\cD_i) \ge H_q(\Pi_v\mu^D) - O(1/q).
\end{equation}

Putting together \eqref{eq:projected-entropies-1}, \eqref{eq:projected-entropies-2} and \eqref{eq:projected-entropies-3} yields the claim.
\end{proof}

We will apply the above proposition to functions $f$ of the form $\phi_x(y)=\tfrac12|x-y|$. Let $\sigma(x,y)=(x-y)/|x-y|\in S^1\subset \R^2$ be the direction generated by $x\neq y$, and note that $Df_x(y)=\sigma(x,y)$. Hence, we have the following corollary of Proposition \ref{prop:projected-entropies-from-local-entropies}.

\begin{corollary} \label{cor:lower-bound-entropy-pinned-dist}
Fix $x\in\R^2$, $D\in\cD_k^{(2)}$, $v\in S^1$ and $q\gg 1$ such that $|\sigma(x,y)-v| \le 2^{-q}$ for all $y\in D$. Then for any $\mu\in\cP([0,1)^2)$ supported on $D$ and any $N\ge q$,
\[
H_N(\phi_x\mu) \ge \int  H_q(\Pi_v\eta) \,d\la \mu\ra_{[0,N]}(\eta) - O(q/N) - O(1/q).
\]
\end{corollary}

\subsection{CP processes}

Following  \cite{Furstenberg08}, we consider CP processes on the tree $([0,1)^d,\rho)$ rather than on Euclidean cubes; the dyadic metric helps avoid technicalities with functions that would not be continuous on Euclidean space (due to dyadic hyperplanes) but are on the tree, notably entropy. We will denote the induced weak topology on $\cP([0,1)^d)$ by $\widetilde{\rho}$, and the weak topology induced by this on $\cP(\cP([0,1)^d))$ by $\widehat{\rho}$. Slightly abusing notation, we will also denote by $\widetilde{\rho}$ the product topology $\widetilde{\rho}\times \rho$ on $\cP([0,1)^d)\times [0,1)^d$, and by $\widehat{\rho}$ the corresponding weak topology on $\cP(\cP([0,1)^d)\times [0,1)^d)$. We note that all these topological spaces are compact and metrizable.  To avoid any ambiguity, we will occasionally denote the topology under consideration with a subscript.

We let $S:[0,1)^d\to [0,1)^d$, $S(x)=2x\bmod 1$ be the doubling map.
\begin{definition}[CP magnification operator]
Let
\[
\Xi = \left\{ (\mu,x)\in \cP([0,1)^d)\times [0,1)^d: x\in\supp_\rho(\mu) \right\}.
\]
We define the \textit{CP magnification operator} $M$ on $\Xi$ by
\[
M(\mu,x)=   (\mu^{x,1}, Sx).
\]
\end{definition}
Note that $M^n(\mu,x)=(\mu^{x,n},S^n x)$.

We now define CP distributions (we refer to probability measures on ``large'' probability spaces such as $\Xi$ as distributions). This definition goes back to \cite{Furstenberg08}; see \cite{HochmanShmerkin12} and \cite{KSS15} for some variants and generalizations.

\begin{definition}[CP distributions]
A distribution $Q$ on $\Xi$ is \emph{adapted}, if there is a disintegration
\begin{equation} \label{eq:adapted}
\int f(\nu,x) \,\mathrm{d}Q(\nu,x) = \iint f(\nu,x)\,\textrm{d}\nu(x) \,\textrm{d}\overline{Q}(\nu),
\end{equation}
for all $f\in C_{\widetilde{\rho}}(\cP([0,1)^d)\times [0,1)^d)$, where $\overline{Q}$ is the projection of $Q$ onto the measure component.

A distribution on $\Xi$ is a \emph{CP distribution} (CPD) if it is $M$-invariant (that is, $MQ=Q$) and adapted.
\end{definition}

Note that adaptedness can be interpreted in the following way: in order to sample a pair $(\mu,x)$ from the distribution $Q$, we have to first sample a measure $\mu$ according to $\overline{Q}$, and then sample a point $x$ using the chosen distribution $\mu$. From now on we will denote by $Q$ both the CPD acting on $\Xi$ and its measure component acting on $\cP([0,1)^d)$, since by adaptedness the latter determines the former.

An easy consequence of the Ergodic Theorem applied to CP distributions is that if $P$ is a CPD which is ergodic under the action of $M$, then $P$-a.e. $\nu$ is exact dimensional, and has dimension
\[
\dim P = \int H_q(\eta)\,dP(\eta)
\]
for any $q\in\N$ (see e.g. \cite[Equation (2.7)]{Furstenberg08}). Let $P = \int P_\mu \, d P(\mu)$ be the ergodic decomposition of $P$ (that is, each $P_\mu$ is $M$-invariant and ergodic, and $\mu\mapsto P_\mu$ is a Borel mapping).  By general properties of Markov processes, $P_\mu$ is again a CPD for $P$-almost all $\mu$, see e.g.  \cite[Remark before Proposition 5.2]{Furstenberg08}. Hence, if $P$ is a (non-necessarily ergodic) CPD, then $P$-a.e. $\nu$ is still exact-dimensional, but $\dim\nu$ needs no longer be $P$-a.e. constant.

\begin{definition}
If $P$ is a CP distribution, we define its \textbf{lower dimension} $\dim_*P$ as the $P$-essential infimum of $\dim\nu$.
\end{definition}

We turn to the behavior of entropy under projections. For this, we recall some results from \cite{HochmanShmerkin12} on CP-processes and projections. Recall that $\Pi_v(x)=\langle v,x\rangle$, $v\in S^1$. Elementary properties of entropy imply that
\begin{equation} \label{eq:continuity-of-projected-entropy}
|H_q(\Pi_v\eta)-H_q(\Pi_{v'}\eta)|  \le O(1/q)  \text{ if } |v-v'|\le 2^{-q},
\end{equation}
with the $O(\cdot)$ constant independent of $\eta$. Indeed, $H_q(\Pi_v\eta)=\frac1q H(\eta, \Pi_v^{-1}\mathcal{D}_q)$ and likewise with $v'$. But if $|v-v'|\le 2^{-q}$, then each element of $\Pi_v^{-1}\mathcal{D}_q$ hits $O(1)$ elements of $\Pi_{v'}^{-1}\mathcal{D}_q$ and vice-versa, so \eqref{eq:continuity-of-projected-entropy} follows.

The following result is a consequence of \cite[Theorem 8.2]{HochmanShmerkin12}. It will act as our projection theorem for entropy.
\begin{theorem} \label{thm:projections-CPD}
Let $P$ be a (not necessarily ergodic) CP-distribution. Write $\cE_q:S^1\to [0,1]$, $v\mapsto \int \min(H_q(\Pi_v\eta),1)\,dP(\eta)$. Then:
\begin{enumerate}
\item The function $\cE_q$ satisfies
\[
|\cE_q(v)-\cE_q(v')| \le O(1/q)  \text{ if } |v-v'|\le 2^{-q};
\]
\item The limit $\cE(v):=\lim_{q\to\infty} \cE_q(v)$ exists for all $v$ and $\cE(v)$ is lower semicontinuous;
\item $\cE(v)\ge \min(\dim_*P,1)$ for almost all $v$;
\end{enumerate}
\end{theorem}
\begin{proof}
The first claim is immediate from \eqref{eq:continuity-of-projected-entropy}. Let
\[
\wt{\cE}_q(v)=\int H_q(\Pi_v\eta)\,dP(\eta).
\]
Since $\Pi_v\eta$ is supported on an interval of length $\le\sqrt{2}$, $|\wt{\cE}_q(v)-\cE_q(v)|\le 1/q$ for all $v\in S^1$. In the case $P$ is ergodic, the latter claims are a particular case of \cite[Theorem 8.2]{HochmanShmerkin12}. More precisely, in \cite{HochmanShmerkin12}, the stated convergence is $\wt{\cE}_q(v)\to \cE(v)$, but by our observation, this immediately yields $\cE_q(v)\to \cE(v)$. The general case follows by considering the ergodic decomposition of $P$ (notice that an integral of lower semicontinuous functions is lower semicontinuous by Fatou's Lemma).
\end{proof}

\subsection{Global tangents}

We want to be able to estimate the entropy of projections of a given measure $\mu\in\cP([0,1)^2)$, but the tools we have at our disposal concern typical measures for a CP process. Following \cite{Hochman13}, we handle this by passing to suitable tangent objects.

Given $\mu\in\cP([0,1)^d)$, the set of accumulation points of $\la\mu\ra_{[0,N]}$ in the $\widehat{\rho}$ metric will be denoted $\mathcal{T}(\mu)$.  Unlike in \cite{Hochman13}, our tangent distributions are global, rather than local but, as the next lemma shows, they are still CP processes:
\begin{lemma} \label{lem:limits-are-CPDs}
Let $\mu_n$ be a sequence in $\cP([0,1)^d)$. Suppose
\[
\la \mu_{N_j}\ra_{[0,N_j]} \underset{\widehat{\rho}}{\to} P,
\]
for some subsequence $(N_j)$. Then $P$ is a CPD (in the sense that the adapted distribution with measure marginal $P$ is a CPD).

In particular, if $\mu\in\cP([0,1)^d)$, then any element of $\mathcal{T}(\mu)$ is a CPD.
\end{lemma}
\begin{proof}
Both the claim and the proof are similar to those of \cite[Propositions 5.2]{Furstenberg08}. For $\nu\in\cP([0,1)^d)$, write
\[
\la \nu \ra^*_{[A,B]} = \frac{1}{B-A}\sum_{n=A}^{B-1} \int \delta(M^n(\nu,x)) d\nu(x).
\]
Note that the measure component of $\la \nu \ra^*_{[A,B]}$ is $\la\nu\ra_{[A,B]}$, and that $\la \nu \ra^*_{[A,B]}$ is always adapted.

Now suppose $\la \mu_{N_j}\ra^*_{[0,N_j]}\to P$ in the $\widehat{\rho}$ topology. Since adaptedness is a closed property (it is tested on equalities of continuous functions), $P$ is adapted.

Since we are using the dyadic metric and $M$ is adapted, $M$ is well defined and continuous at $P$-a.e. $(\mu,x)$ (notice that $x\in\supp_\rho(\mu)$ for $P$-a.e. $(\mu,x)$ by adaptedness). Using standard properties of weak convergence (see e.g. \cite[Theorem 2.7]{Billingsley99}) we conclude that
\begin{align*}
MP &=  M\left(\lim_{j\to\infty} \la \mu_{N_j}\ra^*_{[0,N_j]}\right) \\
&= \lim_{j\to\infty}  M(\la \mu_{N_j}\ra^*_{[0,N_j]})\\
&= \lim_{j\to\infty}  \la \mu_{N_j}\ra^*_{[1,N_j+1]}\\
&= \lim_{j\to\infty}  \la \mu_{N_j} \ra^*_{[0,N_j]} = P.
\end{align*}
\end{proof}

\section{Ahlfors regularity and weak regularity}
\label{sec:Ahlfors-regularity}

The following definition introduces a notion of regularity that, as we will see, extends the concept of Ahlfors-regularity in a suitable sense.
\begin{definition}
\begin{enumerate}
\item A measure $\mu\in\cP([0,1)^d)$ is said to be \textbf{$s$-rich at resolution $(N,q,\delta)$} if
\[
\la \mu \ra_{[0,N]}\{\eta: H_q(\eta) < s-\delta\} < \delta.
\]
\item
A measure $\mu\in\cP([0,1)^d)$ is said to be \textbf{weakly $s$-regular} if for every $\delta>0$ there is $q\in\N$ such that $\mu$ is $s$-rich at resolution $(N,q,\delta)$ for all sufficiently large $N$ (depending on $q$ and $\delta$).
\end{enumerate}
\end{definition}
Note that if a measure is weakly $s$-regular then it is weakly $t$-regular for all $t<s$. In other words, weak $s$-regularity ensures a minimum level of local entropy at most places and scales, but allows for  higher entropy as well.

A first useful feature of weak $s$-regularity is robustness under passing to subsets of positive measure:
\begin{lemma} \label{weak-reg-subsets-of-positive-meas}
If $\mu$ is weakly $s$-regular and $\mu(A)>0$, then $\mu_A$ is weakly $s$-regular.
\end{lemma}
\begin{proof}
This is essentially a consequence of the Lebesgue density theorem (which for the dyadic metric is an immediate consequence of the convergence of conditional expectation given the dyadic filtration). Fix $\delta>0$, and let $q$ be such that $\mu$ is $s$-rich at resolution $(N,q,\delta)$ for all sufficiently large $N$. Write $\Omega_\kappa=\{ \eta: H_q(\eta)>s-\kappa\}$. Then we have
\begin{equation} \label{eq:consequence-s-rich}
\int_B \la \mu,x\ra_{[0,N]}(\Omega_\delta)\,d\mu(x) > \mu(B)-\delta
\end{equation}
for any Borel set $B$, provided $N$ is large enough depending on $\delta$ and $q$ only. By the density point theorem, for $\mu$ almost all $x\in A$, the sequences $\mu_A^{x,n}$ and $\mu^{x,n}$ are $\wt{\rho}$-asymptotic (i.e. $\wt{\rho}(\mu_A^{x,n},\mu^{x,n})\to 0$). In particular, if $N$ is large enough (depending on $\delta$), we then have $\mu_A(B)>1-\delta$, where
\[
B=\left\{ x: \la \mu_A,x \ra_{[0,N]}(\Omega_{2\delta}) \ge  \la \mu,x \ra_{[0,N]}(\Omega_{\delta})\right\}.
\]
Here we used that $H_q$ is continuous on $(\cP([0,1)^d),\widetilde{\rho})$. Recalling \eqref{eq:consequence-s-rich} we conclude that, always assuming $N$ is large enough,
\begin{align*}
\la \mu_A \ra_{[0,N]}(\Omega_{2\delta}) &\ge \frac{1}{\mu(A)}\int_B \la \mu_A,x \ra_{[0,N]}(\Omega_{2\delta})d\mu(x)\\
&\ge \frac{1}{\mu(A)}(\mu(B)-\delta) \ge 1-(1+\mu(A)^{-1})\delta.
\end{align*}
This gives the claim.
\end{proof}

Recall that  $\mu\in\cP(\R^d)$ is called $(s,C)$-Ahlfors regular if  $C^{-1} r^s \le \mu(B(x,r)) \le C r^s$ for all $x\in \supp(\mu)$ and all $r\in (0,1]$. If this holds only for $r\in [2^{-N},1]$, we say that $\mu$ is \textbf{$(s,C)$-Ahlfors regular at scale $2^{-N}$}. We also say that a set $A$ is $(s,C)$-Ahlfors regular if the restriction $\mathcal{H}^s|_A$ is a positive finite $(s,C)$-Ahlfors regular measure.

Given a discrete $(s,C)$-Ahlfors regular set at scale $2^{-N}$ contained in $[0,1]^d$, we can construct a measure $\mu$ in the following manner:
\begin{equation} \label{eq:AR-measure-from-AR-set}
\mu = \mu^A =  \frac{1}{|A|}\sum_{D\in\cD_N}  |A\cap D| \,\mathcal{L}_D,
\end{equation}
where $\mathcal{L}$ denotes $d$-dimensional Lebesgue measure. Reciprocally, from a measure $\mu$ supported on $[0,1]^d$ which is $(s,C)$-Ahlfors regular at scale $2^{-N}$, one can construct the set
\[
A = A^\mu =  \{ x_L(D) : D\in\cD_N, \mu(D)>0 \},
\]
where $x_L(D)$ is the left-endpoint of $D$. One then has the following easy lemma:

\begin{lemma} \label{lem:Ahlfors-set-to-measure}
\begin{enumerate}
\item If $\mu\in\cP([0,1]^d)$ is $(s,C)$-Ahlfors regular at scale $2^{-N}$, then $A^\mu$ is discrete $(s,O(C))$-Ahlfors regular at scale $2^{-N}$.
\item Conversely, if $A\subset [0,1]^d$ is discrete $(s,C)$-Ahlfors regular at scale $2^{-N}$, then $\mu^A$ is $(s,O(C))$-Ahlfors regular at scale $2^{-N}$.
\end{enumerate}
The implicit constants depend only on the ambient dimension $d$.
\end{lemma}
\begin{proof}
Suppose $\mu$ is $(s,C)$-Ahlfors regular at scale $2^{-N}$ and fix $k\in [N]$. If $y\in A^\mu$, then $\mu(B(y,3\cdot 2^{-N}))\in (\Omega(C)2^{-sN},O(C)2^{-sN})$ and likewise with $k$ in place of $N$. If $B(x,2^{-k})\cap A=\{y_1,\ldots,y_m\}$, then $\{ B(y_j,3\cdot 2^{-N})\}$ is a covering of $\supp(\mu)\cap B(x,2^{-k})$ with bounded overlapping, so the first claim follows. The proof of the second claim is analogous, so is omitted.
\end{proof}

We will see that $s$-Ahlfors regular measures are weakly $s$-regular. The following quantitative version of this will be crucial later.
\begin{lemma} \label{lem:Ahlfors-regular-is-rich}
Given $\e, q, N, C$ such that $\log C/q<\e$ and $q<\e N$, the following holds.

Let $\nu$ be $(s,C)$-Ahlfors regular at scale $2^{-N}$. Then if $\mu\in\cP([0,1)^d)$ is supported on $\supp(\nu)$ and $H_N(\mu)> s-\e$, then  $\mu$ is $s$-rich at resolution $(N,q,\sqrt{\e}/C')$, where $C'>0$ depends only on $d$.
\end{lemma}
\begin{proof}
Any constants implicit in the $O$ notation are allowed to depend on $d$ only. Since $q<\e N$ and
\begin{equation} \label{eq:decomp-scales}
\la \mu\ra_{[0,N]} = \frac{N-q}{N}\la \mu\ra_{[0,N-q]} + \frac{q}{N} \la\mu\ra_{[N-q,N]},
\end{equation}
it is enough to show that $\mu$ is $s$-rich at resolution $(N-q,q,\sqrt{\e}/C')$.

Write $A=\supp(\nu)$. We begin by noting that for any $D\in\cD_n^{(d)}$ with $n\in [N-q]$, the set $A$ meets at most $O(C) 2^{sq}$ cubes $D'\subset D, D'\in\cD_{n+q}$. Indeed, let $D_1,\ldots, D_m$ be the sub-cubes of $D$ in $\cD_{n+q}$ that hit $A$, and pick $x_i\in D_i\cap A$. The family $B(x_i,2^{-(n+q)})$ has overlapping bounded by $O(1)$ and each member is contained in $D(2^{-n})$, the $(2^{-n})$-neighborhood of $D$. On the other hand, $D(2^{-n})\subset B(x_1, (\sqrt{d}+1) 2^{-n})$. Hence
\[
O(C) 2^{-sn}  \ge \nu(D(2^{-n}))  \ge \sum_{i=1}^m \nu(B(x_i,2^{-(n+q)})) \ge (O(C))^{-1}  2^{-sn} 2^{-sq} m,
\]
giving the claim. In particular, we see that $H_q(\mu^{x,n}) \le s + O(\log C/q) \le s+O(\e)$ for any $x\in \supp(\mu)$ and any $n\in [N-q]$.

We know from Lemma \ref{lem:local-entropy-from-global-entropy}, the assumption and \eqref{eq:decomp-scales} that
\[
\int H_q(\eta)\, d\la \mu\ra_{[0,N-q]}(\eta) \ge \int H_q(\eta)\, d\la \mu\ra_{[0,N]}(\eta) - \frac{q}{N-q}  \ge s - O(\e)
\]
which, since  $H_q(\eta) \le s+O(\e)$ for $\la \mu\ra_{[0,N-q]}$ a.e. $\eta$, we can rewrite as
\[
\int s-H_q(\eta)+C'\e\, d\la \mu\ra_{[0,N-q]}(\eta) \le O(\e),
\]
where the constant $C'$ was chosen so that the integrand is positive. The lemma now follows from Markov's inequality.
\end{proof}

As an immediate consequence, we deduce that a class of measures, including $s$-Ahlfors regular measures, are indeed weakly $s$-regular.
\begin{corollary}
If $\mu$ is supported on an $s$-Ahlfors regular set and $\ledim\mu=s$, then $\mu$ is weakly $s$-regular. In particular, this is the case for $\nu_A$ when $\nu$ is $s$-Ahlfors regular and $\nu(A)>0$.
\end{corollary}
\begin{proof}
Fix $\delta>0$ and take $q$ large enough that $\log(C)/q < \delta^2$. Since $\ledim\mu=s$, we know that $H_N(\mu)>s-\delta^2$ for large enough $N$. If $N$ is also large enough that $N > \delta^{-2} q$, then the previous lemma says that $\mu$ is rich at resolution $(N,q,O(\delta))$.

For the latter claim, note that $\nu_A$ has exact dimension $s$ (as a consequence of the density point theorem), so that $\ledim\nu_A=s$.
\end{proof}

\section{Proof of Theorem \ref{thm:many-large-pinned-dist-sets}, and consequences}
\label{sec:pinned-dist-sets}

\subsection{Discrete conical density lemmas}
\label{subsec:conical-density}

In the proof of Theorem \ref{thm:many-large-pinned-dist-sets} we will require some discrete conical density results. These are similar to those in \cite[Section 3]{SSS13}.

We say that a set $A\subset [0,1]^2$ is \textbf{$k$-discrete} if $|A\cap D|\le 1$ for all $D\in\cD_k^{(2)}$. Also, let $X(a,\beta,v)$ be the two-sided cone with center $a\in\R^2$, opening $\beta\in (0,\pi/2)$ and direction $v\in S^1$. The following is a discrete analog of \cite[Lemma 15.13]{Mattila95}.
\begin{lemma} \label{lem:discrete-dense-radial-projection}
Given $\beta>0$, there is a constant $C=C(\beta)>0$ such that the following holds. If $A$ is $k$-discrete and for each $a\in A$ there is a direction $v$ such that
\[
X(a,\beta,v) \cap A\setminus \{a\} = \emptyset,
\]
then $|A|\le C 2^k$.
\end{lemma}
\begin{proof}
We begin with a simplification. Choose a finite set $\{ v_j\}$ with $O_\beta(1)$ elements such that for every $v\in S^1$ there exists $v_j$ with
\[
X(a,\beta/2,v_j)\subset X(a,\beta,v).
\]
Hence, if $A$ is as in the statement, for every $a\in A$ we can pick $j(a)$ such that
\[
X(a,\beta/2,v_{j(a)}) \cap A\setminus\{a\}= \emptyset.
\]
Let $A_j = \{ a\in A: j(a)=j\}$. Some $A_j$ has $\ge |A|/O_\beta(1)$-elements. Moreover, by passing to a further refinement with $|A|/O_\beta(1)$ elements, we can assume that the elements of $A_j$ are $(2^{-k})$-separated. This shows that it is enough to prove the following statement: if $v_0$ is a fixed direction, and $A\subset [0,1]^2$ is a $(2^{-k})$-separated set such that
\[
X(a,\beta/2,v_0) \cap A\setminus\{a\}= \emptyset \quad\text{for all }a\in A,
\]
then $|A|\le O_\beta(2^k)$.

Let $\Pi(x)=\Pi_{v_0^\perp}(x)= x \cdot v_0^{\perp}$, where $v_0^\perp$ is a unit vector perpendicular to $v_0$. It follows from our assumptions on $A$ that $|\Pi(a)-\Pi(a')| \sin(\beta/2)\ge 2^{-k}$ for any distinct $a,a'\in A$. In particular, $\Pi|_A$ is injective and its range has $O_\beta(2^k)$ elements so $|A| \le O_\beta(2^k)$, as claimed.
\end{proof}

For sets which are dense in a discrete $s$-Ahlfors regular set, we obtain the following consequence.
\begin{lemma} \label{lem:dense-radial-proj-discr-AR}
Given $s\in (1,2), C>1, \kappa\in (0,(s-1)/(2s)), \beta\in (0,\pi/2)$, the following holds for all large enough $N$ (depending on $s,C,\kappa,\beta$ only):

Let $A$ be a discrete $(s,C)$-Ahlfors regular set at scale $2^{-N}$, and suppose $B\subset A$ satisfies $|B|> 2^{(1-\kappa) s N}$. Then there exists a subset $E\subset B$ with $|E|\le 2^{(1-\kappa)sN}$ such that for all $x\in B\setminus E$ and any $v\in S^1$ there exists $y\in B$ such that $y\in X(x,\beta,v)$ and $|x-y|\ge 2^{-2s(s-1)^{-1}\kappa N}$.
\end{lemma}
\begin{proof}
Write $\kappa'=2s(s-1)^{-1}\kappa$ and note that $\kappa'\in (0,1)$. We say that a point $x\in B$ is \emph{well surrounded} if for every $v\in S^1$ there is $y\in B$ such that $y\in X(x,\beta,v)$ and $|x-y|\ge 2^{-\kappa' N}$.

Let $E\subset B$ be the set of all points in $B$ which are \emph{not} well surrounded, and suppose $|E|>2^{(1-\kappa)sN}$. Let $E_1$ be a maximal $(2^{-\kappa' N})$-separated subset of $E$. Since each ball of radius $2^{-\kappa' N}$ contains $O(C) 2^{(1-\kappa')sN}$ points of $A\supset E$, it follows that $|E_1|> \Omega(C) 2^{(\kappa'-\kappa)s N}$. Note that $(\kappa'-\kappa)s> \kappa$, and let $C'=C'(\beta)$ be the constant given by Lemma \ref{lem:discrete-dense-radial-projection}. Provided $N$ is large enough that $\Omega(C) 2^{(\kappa'-\kappa)s N}> C' 2^{\kappa N}$, it follows from Lemma \ref{lem:discrete-dense-radial-projection} and the definitions that $E_1$ contains a well surrounded point. This contradiction proves the lemma.
\end{proof}

\subsection{Pinned distance sets in discrete regular sets}

The core of the proof of Theorem \ref{thm:many-large-pinned-dist-sets} consists in showing the existence of \emph{one} large pinned distance set. We state and prove the corresponding statement separately:
\begin{prop} \label{prop:discrete-Falconer}
Given $s>1, C>1, t\in (0,1)$, there exist $\e=\e(s,C,t)>0$ and $N_0=N_0(s,C,t,\e)\in\N$ such that the following holds: if $N\ge N_0$, and $A\subset[0,1]^2$ is a subset of a discrete $(s,C)$-Ahlfors regular set at scale $2^{-N}$, such that $|A| \ge 2^{(s-\e)N}$, then there exists $x\in A$ such that
\[
\cN(\dist(x,A),2^{-N}) \ge 2^{tN}.
\]
\end{prop}
Before embarking on the proof of this proposition, we show how to deduce Theorem \ref{thm:many-large-pinned-dist-sets} from it.

\begin{proof}[Proof of Theorem \ref{thm:many-large-pinned-dist-sets} (assuming Proposition \ref{prop:discrete-Falconer})]
Let $\e$ and $N_0$ be as given by Proposition \ref{prop:discrete-Falconer}, and take $N\ge N_0$. Let
\[
B = \{ x\in A: \mathcal{N}(\dist(x,A),2^{-N})<2^{tN}\}.
\]
In particular, if $x\in B$, then  $\mathcal{N}(\dist(x,B),2^{-N})<2^{tN}$. By Proposition \ref{prop:discrete-Falconer} applied to $B$, $|B|\le 2^{(s-\e)N}$, as claimed.
\end{proof}

The rest of this section is devoted to the proof of Proposition \ref{prop:discrete-Falconer}. Suppose the claim is false. Then  we can find sequences $N_j\to\infty$, $\e_j\to 0$, $A_j, B_j\subset [0,1]^2$ such that $B_j$ is discrete $(s,C)$-Ahlfors regular at scale $2^{-N_j}$, $A_j\subset B_j$, $|A_j|\ge 2^{(s-\e_j)N_j}$, and $\cN(\dist(x,A_j),2^{-N_j}) < 2^{tN_j}$ for all $x\in A_j$.
Let
\[
\mu_j = \frac{1}{|A_j|} \sum_{D\in\mathcal{D}_{N_j}} |A_j\cap D|\mathcal{L}_D.
\]

By passing to a subsequence if needed, we may assume that $\la\mu_j\ra_{[0,N_j]}$ converges to some $P\in\cP(\cP([0,1)^2))$ in the $\widehat{\rho}$ topology. By Lemma \ref{lem:limits-are-CPDs}, $P$ is a CPD. We underline that $P$ needs not be ergodic under $M$; if it was, the next lemma would hold automatically. It is only in this lemma that the hypothesis of Ahlfors-regularity gets used.

\begin{lemma} \label{lem:lower-dim-regular-CPD}
$\dim_*P\ge s$
\end{lemma}
\begin{proof}
Since $P$-a.e. measure is exact-dimensional, it is enough to show that $\uedim\nu\ge s$ for $P$-a.e. $\nu$. In turn, by Borel-Cantelli this will follow if we can show that for every $\wt{\delta}>0$, if $q$ is sufficiently large (depending on $\wt{\delta}$), then
\begin{equation} \label{eq:supported-on-large-measures}
P\{\eta: H_q(\eta) \ge s-\wt{\delta}\} > 1-\wt{\delta}.
\end{equation}

Since $\e_j\to 0$, we know that $|A_j| \ge 2^{(s-\delta/2) N_j}$ for large enough $j$. Since $B_j$, and hence $A_j$, hits at most $C$ points in each dyadic square of side length $2^{-N_j}$, a calculation shows that, if $j$ is large enough, then
\[
H_{N_j}(\mu_j) \ge s-\delta.
\]
Lemma \ref{lem:Ahlfors-regular-is-rich} (together with Lemma \ref{lem:Ahlfors-set-to-measure}(2) applied to $B_j$) can then be invoked to conclude that, given $q$ is taken large enough in terms of $\delta$, and then $j$ is taken large enough in terms of $q$,  the measure $\mu_j$ is $s$-rich at resolution $(N_j,q,\sqrt{\delta}/C')$ for some universal $C'>0$. Since $H_q$ is continuous on $(\mathcal{P}([0,1)^d),\widetilde{\rho})$, the set $\{\eta: H_q(\eta) \ge s-\sqrt{\delta}/C'\}$ is compact, so we can pass to the limit to obtain our claim \eqref{eq:supported-on-large-measures}.
\end{proof}

At this point we fix a small $\delta>0$. In the end, a contradiction will be obtained provided $\delta$ was taken sufficiently small.

Let $\cE_q, \cE$ be the functions given by Theorem \ref{thm:projections-CPD}. We fix $v\in S^1$ such that $\cE(v)= 1$ (this is possible because $\dim_*P>1$).  Pick $q$ large enough that $1/q\le \delta$ and (recalling the definition of $\cE_q$)
\begin{equation} \label{eq:expected-projected-entropy-large}
\int \min(H_q(\Pi_v\eta),1)\,dP(\eta) > 1-\delta.
\end{equation}

We take $j$ large enough that $|A_j| \ge 2^{(1-\delta/2)s N_j}$. We know from Lemma \ref{lem:dense-radial-proj-discr-AR} that, again assuming $j$ is large enough, there is a set $E_j$ with $|E_j| \le 2^{(1-\delta)s N_j}$ such that if $x\in A_j\setminus E_j$, then there is $y\in A_j$ with $y\in X(x,2^{-q-1},v)$ and $|x-y|\ge 2^{-K\delta N_j}$, with $K>0$ depending only on $s$.

Write $M_j = \lfloor (K+1)\delta N_j\rfloor$ and note that if $y\in X(x,2^{-q-1},v)$ and $|x-y|\ge 2^{-K\delta N_j}$, then $y\in X(x',2^{-q},v)$ for all $x'\in D_{M_j}(x)$, again provided $j$ is large enough (the point is that the diameter of $D_{M_j}(x)$ is very small compared to $|x-y|$). If $D_{j,1},\ldots, D_{j,L_j}\in \cD_{M_j}^{(2)}$ is an enumeration of the the squares containing some point of $A_j\setminus E_j$, the previous observations show that, if $j$ is sufficiently large, then:
\begin{enumerate}
\item For each $k\in\{1,\ldots, L_j\}$, there is $y_{j,k}\in A_j$ such that $y_{j,k}\in X(x,2^{-q},v)$ for all $x\in D_{j,k}$.
\item
\begin{equation} \label{eq:good-projections-bad-set}
\mu_j\left(\bigcup_{k=1}^{L_j} D_{j,k}\right) > 1 - O_C(2^{-\delta N_j}) > 1-\delta.
\end{equation}
\end{enumerate}

Since $P$-a.e. measure is exact dimensional and has dimension $>1$, $P$-a.e. measure gives no mass to lines, hence the function $\eta\mapsto H_q(\Pi_v\eta)$ is continuous $P$-almost everywhere. Consequently, if $j$ is large enough we deduce from \eqref{eq:expected-projected-entropy-large} that
\[
\int \min(H_q(\Pi_v\eta),1) d\la \mu_j \ra_{[0,N_j]}(\eta) > 1-\delta.
\]
Since $M_j/N_j\le (K+1)\delta$, it follows that
\begin{equation} \label{eq:good-projections-global}
\int \min(H_q(\Pi_v\eta),1) d\la \mu_j \ra_{[M_j,N_j]}(\eta) > 1-(K+2)\delta.
\end{equation}
On the other hand, note that for any $\eta\in\cP([0,1)^2)$, and any $1\le M\le n$, there is a decomposition
\[
\la\eta\ra_n = \sum_{D\in\mathcal{D}_M} \eta(D) \la\eta_D\ra_n.
\]
Hence, if we denote $\nu_{j,k}=(\mu_j)_{D_{j,k}}$, adding up over $n=M_j,M_{j+1},\ldots, N_j$ yields
\begin{equation} \label{eq:good-projections-split-into-squares}
\la \mu_j \ra_{[M_j,N_j]} = \sum_{k=1}^{L_j}  \mu(D_{j,k}) \la \nu_{j,k} \ra_{[M_j,N_j]} + Q,
\end{equation}
where $Q$ has total mass at most $\delta$ by \eqref{eq:good-projections-bad-set}.

It follows from \eqref{eq:good-projections-global} and \eqref{eq:good-projections-split-into-squares} that
for large enough $j$ there exists a square $D_{j,k}$ with
\[
\int \min(H_q(\Pi_v\eta),1) d\la \nu_{j,k} \ra_{[M_j,N_j]}(\eta) > 1-(K+3)\delta.
\]
From now on we fix such a good square $D_{j,k}$ for each $j$, denote it simply by $D_j$ and forget about the other squares. We also denote $\nu_j=\nu_{j,k}$ and $y_j=y_{j,k}$. Recall that this is the point in $A_j$, whose existence we established earlier, such that $y_j \in X(x,2^{-q},v)$ for all $x\in D_j$.

Using again that  $M_j/N_j\le (K+1)\delta$, we get
\[
\int H_q(\Pi_v\eta) d\la \nu_{j} \ra_{[0,N_j]}(\eta) > 1-(2K+4)\delta.
\]
We have arranged things so that the hypotheses of Corollary \ref{cor:lower-bound-entropy-pinned-dist} are met. Since $1/q<\delta$, we conclude that, provided $j$ is large enough that $q/N_j<\delta$,
\[
H_{N_j}(\phi_{y_j}\nu_j) \ge 1 -O_s(\delta),
\]
where $\phi_y(x)=\frac12|x-y|$. In particular, since $\nu_j$ is supported on a $(2^{-N_j})$-neighborhood of $A_j$, this shows that
\[
\cN(\dist(y_j,A_j),2^{-N_j}) \ge 2^{(1-O_s(\delta))N_j}
\]
provided $j$ is large enough (depending on $\delta$). This contradicts with
\[
\cN(\dist(y_j,A_j),2^{-N_j}) < 2^{t N_j} \text{ for all $j$}
\]
if $\delta$ is small enough, yielding the result.

\subsection{Proof of Corollaries \ref{cor:pinned-dist-set-large-dev} and \ref{cor:dist-set-box-dim}}

It is now easy to deduce Corollaries \ref{cor:pinned-dist-set-large-dev} and \ref{cor:dist-set-box-dim}.

\begin{proof}[Proof of Corollary \ref{cor:pinned-dist-set-large-dev}]

Let $A$ be as in the statement. Write $A_N$ for the collection of centers of squares in $\mathcal{D}_N$ hitting $A$, so that in particular $A$ is contained in the $2^{-N}$-neighborhood of $A_N$. By Lemma \ref{lem:Ahlfors-set-to-measure}, the sets $A_N$ are contained in a $(s,C')$-discrete Ahlfors regular set at scale $2^{-N}$, for some $C'=O(C)$.

Let $\e=\e(s,C',t)>0$ be the value given by Theorem \ref{thm:many-large-pinned-dist-sets}. By the theorem, if $N$ is large enough, then there is a set $B_N\subset A_N$ with $|B_N|< 2^{(s-\e)N}$ such that if $x\in A_N\setminus B_N$, then
\[
\cN(\dist(x,A_N),2^{-N}) \ge 2^{tN}.
\]
Let
\[
B=\limsup_N B_N(2^{-N}) = \bigcap_{N=1}^\infty \bigcup_{M=N}^\infty B_M(2^{-M}).
\]
Fix $s'>s-\e$. Since $|B_M|< 2^{(s-\e)M}$, we see that for each $N$ the set $B$ can be covered by a sequence of balls containing $2^{(s-\e)M}$ balls of radius $2^{-M}$ for each $M\ge N$. It follows that $\mathcal{H}^{s'}(B)<0$ so that, letting $s'\downarrow s-\e$, we get $\hdim(B)\le s-\e$.

On the other hand, it follows from the previous observations that if $x\in A\setminus B$, then
\[
\cN(\dist(x,A),2^{-N}) \ge 2^{tN} \quad\text{for large enough } N,
\]
so $\lbdim(\dist(x,A))\ge t$. This gives the first claim.

Now suppose $\mathcal{H}^s(A)>0$. It is enough to check that if $t\in (0,1)$, then
\[
\lbdim(\dist(x,A))\ge t
\]
for $\mathcal{H}^s|_A$-almost all $x$. Suppose otherwise. Then there is a set $B\subset A$ such that $\mathcal{H}^s(B)>0$ (in particular, $\hdim(B)\ge s$), and $\lbdim(\dist(x,B))<t$ for all $x\in B$. This contradicts the first claim.
\end{proof}

\begin{remark}
In fact, $\hdim$ can be replaced by $\mlbdim$ in the first part of the corollary above - the proof is identical. Moreover, a small variant of the proof shows that, under the same assumptions,
\[
\pdim\{x\in A: \ubdim(\dist(x,A))<t\} < s-\e.
\]
\end{remark}

\begin{proof}[Proof of Corollary \ref{cor:dist-set-box-dim}]
We give the proof for $\lbdim$, the proof for $\ubdim$ is almost identical. As in the proof of Corollary \ref{cor:pinned-dist-set-large-dev}, we let $A_N$ be the $(2^{-N})$-discretization of $A$, so that $A_N$ is contained in the $(2^{-N})$-neighborhood of a $(s,C')$-discrete Ahlfors-regular set with $C'=O(C)$. Fix $t\in  (0,1)$; it is enough to show that $\lbdim(\dist(A,A))\ge t$. Let $\e=\e(s,C',t)>0$ be the number given by Theorem \ref{thm:many-large-pinned-dist-sets}. If $N$ is large enough, $|A_N|\ge 2^{(s-\e)N}$, so Theorem \ref{thm:many-large-pinned-dist-sets} says that there is $x=x_N\in A_N$ such that $\mathcal{N}(\dist(x,A_N),2^{-N}) \ge 2^{tN}$. But
\[
\mathcal{N}(\dist(A,A), 2^{-N}) \ge \frac13 \mathcal{N}(\dist(x,A_N),2^{-N})
\]
by the triangle inequality, so the claim follows.
\end{proof}

\section{Distances between two sets}
\label{sec:distances-between-sets}

Now we investigate the set of distances between two sets. The following result immediately implies Theorem \ref{thm:mlbdim-distance-sets-AR}.
\begin{theorem} \label{thm:mlbdim-distance-sets}
Let $A,B\subset \R^d$ be two Borel sets such that $\hdim(A)>1$ and $\mu(B)>0$ for some weakly $1$-regular $\mu$ which also satisfies $\hdim(\mu)>1$. Then
\[
\mlbdim(\dist(A,B)) = 1.
\]
\end{theorem}

We begin the proof of Theorem \ref{thm:mlbdim-distance-sets}, by showing that it is enough to prove the corresponding claim for lower box counting dimension.

\begin{lemma} \label{lem:simplif-box-dim}
Suppose that, under the assumptions of Theorem \ref{thm:mlbdim-distance-sets},
\begin{equation} \label{eq:lbd-distance-set}
\lbdim(\dist(A,B)) = 1.
\end{equation}
Then Theorem \ref{thm:mlbdim-distance-sets} holds.
\end{lemma}
\begin{proof}
Let $A,B$ be as in the statement of Theorem \ref{thm:mlbdim-distance-sets}. Without loss of generality, $A$ and $B$ can be taken to be compact. Moreover, by Frostman's Lemma, we may further assume that $\hdim(A\cap B(x,r))>1$ for any open ball $B(x,r)$ for which $A\cap B(x,r)\neq \emptyset$ (more precisely, let $\nu$ be a measure supported on $A$ such that $\nu(B(x,r)) \le C\, r^s$ for some $s>1$, and replace $A$ by the support of $\mu$). Finally, we may assume that $\supp(\mu)=B$ simply by replacing $B$ by $\supp(\mu_B)$.

After these reductions, suppose $\mlbdim(\dist(A,B))=t<1$, and partition $\dist(A,B)$ into countably many Borel sets $D_j$, so that $\lbdim(D_j)\le t$ for all $j$. By Baire's Theorem (and since we are assuming that $A$ and $B$ are compact), $\dist^{-1}(D_j)$ has nonempty interior in $A\times B$ for some $j$. Hence $\dist^{-1}(D_j)$ contains a set of the form $A_0\times B_0$ where, by our assumptions, $\hdim(A_0)>1$ and $\mu(B_0)>0$. This contradicts \eqref{eq:lbd-distance-set}.
\end{proof}

Recall that the direction determined by two different vectors $x,y\in\R^2$ is denoted by $\sigma(x,y)$. In the next Lemma we perform a further regularization of the set $B$; this step uses a recent result of Mattila and Orponen \cite{MattilaOrponen15}.
\begin{lemma} \label{lem:simplif-uniform-angle}
In order to prove Theorem \ref{thm:mlbdim-distance-sets}, it is enough to prove the following.

Let $A, B,\mu$ be as in the statement of the theorem, and further assume that $A, B$ are compact and disjoint and that there exists a set $\Theta\subset S^1$ of positive length such that for each $v\in \Theta$,
\[
\mu_B\{y: \sigma(x,y)=v \text{ for some }x\in A\} > 1-\delta,
\]
for some $\delta\in (0,1)$. Then
\[
\lbdim(\dist(A,B)) > 1-\e(\delta),
\]
where $\e(\delta)\downarrow 0$ as $\delta\downarrow 0$.
\end{lemma}
\begin{proof}
Suppose there exist $A,B,\mu$ as in Theorem \ref{thm:mlbdim-distance-sets} with
\[
\lbdim(\dist(A,B)) < 1.
\]
In light of Lemma \ref{lem:simplif-box-dim}, to derive a contradiction it is enough to show that, given $\delta>0$, we can find subsets $A_0, B_0$ of $A, B$ (depending on $\delta$), so that the pair $(A_0,B_0)$ satisfies the assumptions in the present lemma.

We start by noticing that we can easily make $A, B$ disjoint by taking appropriate subsets so we assume that they are already disjoint as given. By \cite[Corollary 1.5]{MattilaOrponen15}, for $\mu$-almost every $y\in B$, the set $\Theta_y=\{ \sigma(x,y):x\in A\}$ has positive length. Notice that the set \[
\Upsilon=\{ (v,y):v\in\Theta_y\}
\]
is Borel (we leave the routine verification to the reader). Thus, by Fubini, $(\gamma\times \mu)(\Upsilon)>0$ (where $\gamma$ is Lebesgue measure on $S^1$). Let $(v_0,y_0)$ be a $(\gamma\times \mu)$-density point of $\Upsilon$ (for its existence, see e.g. \cite[Corollary 2.14]{Mattila95}). We can then find compact neighborhoods $\Theta_0$ of $v_0$ and $B_0$ of $y_0$,  such that
\[
(\gamma\times \mu)\{ (v,y)\in \Theta_0\times B_0\cap \Upsilon\} \ge (1-\delta/2)\gamma(\Theta_0)\mu(B_0).
\]
Applying Fubini once again, we conclude that for $v$ in a set $\Theta$ of positive measure (contained in $\Theta_0$),
\[
\mu\{ y\in B_0: v\in \Theta_y\} > (1-\delta)\mu(B_0) .
\]
Replacing $B$ and $\mu$ by $B_0$ and $\mu_{B_0}$ concludes the proof.
\end{proof}

\begin{proof}[Proof of Theorem \ref{thm:mlbdim-distance-sets}]
We will prove the claim of Lemma \ref{lem:simplif-uniform-angle} with $\e(\delta)=O(\delta)$. Hence, let $A, B, \mu, \Theta$ and $\delta$ be as in that lemma. We also assume, as we may, that $B\subset [0,1)^2$.

Let $t=\lbdim(\dist(A,B))$. Our goal is then to show that $t>1-O(\delta)$. Recall that $\mathcal{N}(X,2^{-N})$ stands for the number of cubes in $\cD_N$ hit by the set $X$. Let $N_j\to\infty$ be a sequence such that
\begin{equation} \label{eq:convergence-of-box-counting}
\frac{\log\mathcal{N}(\dist(A,B),2^{-N_j})}{N_j} \to t.
\end{equation}
By passing to a subsequence if needed, we may assume that $\la \mu_B \ra_{[0,N_j]}$ converges, in the $\wh{\rho}$ topology, to a distribution $P$ which, as we have seen in Lemma \ref{lem:limits-are-CPDs}, is a CPD. Moreover, using weak $1$-regularity of $\mu$, the same argument from Lemma \ref{lem:lower-dim-regular-CPD} shows that $\dim_*P \ge 1$.

Let $\cE_q$ and $\cE$ be as in Theorem \ref{thm:projections-CPD}. By the last part of that theorem, we know that $\cE(v)=1$ for almost all $v$. Thus, since $\Theta$ has positive measure, we can fix $v$ such that $\cE(v)=1$ and $v\in\Theta$.

From this point on, the proof is similar to that of Proposition \ref{prop:discrete-Falconer} but simpler as we do not need quantitative estimates. Since $\cE_q\to \cE$ pointwise, we can fix $q=q(P,\delta)$ such that $\cE_q(v)>1-\delta^2$ and $1/q<\delta$. Recalling the definition of $\cE_q$, we see from Markov's inequality that
\[
P(\{\eta: H_q(\Pi_v\eta)>1-\delta\})>1-\delta.
\]

Now since $A$ and $B$ are compact and disjoint, there exists $k$ (depending on $A,B,q$) such that if $x\in A, y\in B$ and $\sigma(x,y)=v$, then
\begin{equation} \label{eq:direction-almost-constant-on-square}
|\sigma(x,y')-v| \le 2^{-q} \quad\text{if } y'\in D_k(y).
\end{equation}
Next, let $B_0$ be the union of $D_k(y)$ over all $y$ such that $\sigma(x,y)=v$ for some $x\in A$. Note that $\mu(B_0)>1-\delta$ by hypothesis. Let $D_1,\ldots, D_\ell$ be the cubes in $\cD_k$ that make up $B_0$, and pick $y_i\in D_i, x_i\in A$ such that $\sigma(x_i,y_i)=v$ (i.e. if there are many such pairs we select one; this can be done in a Borel manner although we do not require this). Arguing exactly as in the proof of Proposition \ref{prop:discrete-Falconer}, for each sufficiently large $j$ we find a cube $D_{i}$ (with $i$ depending on $j$) such that
\begin{equation} \label{eq:cube-with-good-projections}
\la \mu_{D_{i}} \ra_{[0,N_j]}(\{\eta: H_q(\Pi_v\eta)>1-2\delta\})\ge 1-O(\delta).
\end{equation}
Hence, there is a value of $i$ such that the above happens infinitely often. From now on we fix that value of $i$, and write $M_j\to\infty$ for the corresponding subsequence of $N_j$.

Write $\phi_x(y)=\frac12|x-y|$. It follows from \eqref{eq:direction-almost-constant-on-square} and Corollary \ref{cor:lower-bound-entropy-pinned-dist} that if $j$ is large enough, then
\[
H_{M_j}(\phi_{x_i}\mu_{D_i}) \ge 1-O(\delta).
\]

Since $\phi_{x_i}\mu_{D_i}$ is supported on $\frac12\dist(A,B)$ and $M_j$ is a subsequence of $N_j$, we conclude from \eqref{eq:convergence-of-box-counting} that
\[
t=\lbdim(\dist(A,B))> 1-O(\delta),
\]
which is what we wanted to show.
\end{proof}

\begin{remark}
The proof of \cite[Corollary 1.5]{MattilaOrponen15} goes through under the assumption of positive $1$-capacity rather than Hausdorff dimension $>1$ (or finite $I_1$ energy for the corresponding statement for measures that occurs in the proof). Hence, the assumptions in Theorem \ref{thm:mlbdim-distance-sets} can be weakened to positive $1$-capacity of $A$ and $I_1(\mu)<+\infty$ instead of $\hdim\mu>1$ (we still need to assume that $\mu$ is weakly $1$-regular).  This gives many examples of (pairs of) sets of dimension $1$ to which the results apply.
\end{remark}

\begin{proof}[Proof of Corollary \ref{cor:pinned-dist-set-upper-box-dim}]
Let $A_0= \{ x\in A:\ubdim(\dist(x,B))=1\}$. The proof of Theorem \ref{thm:mlbdim-distance-sets} shows that $A_0$ is nonempty (we begin with a sequence $N_j\to\infty$ such that $\la \mu_B \ra_{[0,N_j]}$ converges; the rest of the proof is identical). This implies that $\hdim(A\setminus A_0)\le 1$, for otherwise there would be $x\in A\setminus A_0$ such that $\ubdim(\dist(x,B))=1$.
\end{proof}

\section*{Acknowledgments}

I thank Tuomas Orponen for useful discussions, and two anonymous referees for a careful reading and many helpful suggestions that have made this a better paper.

\bibliographystyle{amsplain}


\begin{dajauthors}
\begin{authorinfo}[ps]
  Pablo Shmerkin\\
  Torcuato Di Tella University and CONICET\\
  Buenos Aires, Argentina\\
  pshmerkin\imageat{}utdt\imagedot{}edu \\
  \url{http://www.utdt.edu/profesores/pshmerkin}
\end{authorinfo}
\end{dajauthors}

\end{document}